\documentclass[11pt,a4paper,reqno]{amsart}

\usepackage{amsmath, amsthm, amsfonts, amssymb}
\usepackage{graphicx, hyperref, bm,verbatim,siunitx}
\usepackage[english]{babel}
\usepackage[T1]{fontenc}
\usepackage[ansinew]{inputenc}
\usepackage{lmodern} %Type1-font for non-english texts and characters
\usepackage{latexsym}
\usepackage{setspace}
\usepackage{float}
\usepackage{enumerate}
\usepackage[numbers]{natbib}

\theoremstyle{plain}
\newtheorem{theorem}{Theorem}
\newtheorem{lemma}[theorem]{Lemma}
\newtheorem{proposition}[theorem]{Proposition}
\newtheorem{corollary}[theorem]{Corollary}

\newtheorem{conjecture}{Conjecture}

\newtheorem{definition}{Definition}
\newtheorem{remark}{Remark}
\numberwithin{equation}{section}
\begin{document}
%% ---------------------------------------------------------
\title{On the Collatz general problem $qn+1$}
%% ---------------------------------------------------------

\author{Robert Santos}

\email{rsants@gmail.com}               

%% --------------------------------------------------------------------------------------------
\begin{abstract}
In this work the generalized Collatz problem $qn+1$ ($q$ odd) is studied. As a natural generalization of the original $3n+1$ problem, it consists of a discrete dynamical system of an arithmetical kind. Using standard methods of number theory and dynamical systems, general properties are established, such as the existence of finitely many periodic sequences for each $q$. In particular, when $q$ is a Mersenne number, $q=2^p-1$, there only exists one such cycle, known as the trivial one. Further analysis based on a probabilistic model shows that for $q=3$ the asymptotic behavior of all sequences is always convergent, whereas for $q\geq 5$ the asymptotic behavior of the sequences is divergent for almost all numbers (for a set of natural density one). This leads to the conclusion that the so called Collatz Conjecture is true, and that $q=3$ is a very special case among the others (Crandall conjecture). Indeed, it is conjectured that the general problem $qn+1$ is undecidable.
\end{abstract}

%\keywords{Collatz conjecture, 3x+1 problem, generalized Collatz problem, discrete dynamical system, Diophantine exponential equations, undecidability}
% MSC Classification{37P99 \and 37-02}

\maketitle
%% --------------------------------------------------------------------------------------------

%-------------------------------------
%-------------------------------------
\section{Introduction}\label{sec:intro}
%-------------------------------------
%-------------------------------------
The Collatz conjecture, also known as the $3n+1$ problem \footnote{Usually, it is denoted by the $3x+1$ problem, but following Wirsching \cite{wirsching_dynamical_1998}, we rather prefer the $3n+1$ nomenclature to highlight that we deal with positive integer numbers.}, is a famous mathematical problem stated in the 1930's by Lothar Collatz \cite{Collatz_motivation_1986}. It asserts that given any positive integer number $n$, the algorithm that iteratively computes $(3n+1)/2$ if $n$ is odd and $n/2$ if $n$ is even, produces a sequence of numbers that always reaches $1$. Without exceptions. And numerical experiments show that this is true for all $n \leq 2^{68}$ \cite{barina_convergence_2020,roosendaal_3x_nodate,Oliveira.e.Silva-2010-1-EVCC}. The interested reader will find a wide exposition about the history of the problem in \cite{wirsching_dynamical_1998,lagarias_ultimate_2011,chamberland_update_2003}.

There are few theoretical results concerning the $3n+1$ problem. Among the most strong, it has been proven that the number of integer values in $[1,x]$ such that the conjecture is true is greater than $x^{0.84}$ for a sufficiently large $x$ \cite{krasikov_bounds_2003}. The approach of the authors uses a non linear programming method to obtain lower bounds for solutions of a system of functional difference inequalities (Krasikov's inequalities) associated to the $3n+1$ problem. They also expected that an improvement of the form $x^{1-\epsilon}$ for any $\epsilon>0$ could be reached with intensive large computations, although being exponentially consuming. But no better results has been published since then. 

Also, there are a number of results in the \emph{almost all} setting, beginning with the earliest work by Terras \cite{terras_stopping_1976} to the most recent result by Tao \cite{tao_almost_2019}. Informally speaking, they state that for almost all numbers, in the sense of a density one set, the infimum of the sequence produced by the algorithm is bounded by a function that depends on the initial input. Unfortunately, these results neither rule out the existence of a (very exotic) divergent sequence nor says nothing about the location of periodic behaviors.

The study of a generalization of the Collatz problem, say $qn+1$ with $q$ any odd number, has been reveled very fruitful to understand the original $3n+1$ problem, and also posses the background needed to show why the case $q=3$ is such an special case among the others. That is the problem we analyze in this paper. There are also important previous results on it. Crandall was one of the first who worked on the $qn+1$ problem \cite{crandall_3x1_1978}, and formulated the conjecture that for $q\geq 5$ there is, at least, one sequence that never visits 1. He proved its validity for the cases $q=5$, $q=181$, and $q=1093$ (extensible also to $q=3511$ as pointed out by Lagarias \cite{lagarias20033x1}), and gave an heuristic probabilistic argument for the remaining cases. Along this line, Matthews and his collaborators \cite{matthews_generalized_2010} extended the conjecture to generalized Collatz maps for which the $qn+1$ is an special case. They used Markov chains to show its plausibility, and also showed that the extension of these maps to the complete set of $d$-adic integers $\mathbb Z_d$, $d\geq 2$, are ergodic and, consequently, satisfy the conjecture. Unfortunately, this does not solve the $qn+1$ problem since the set of positive integers $\mathbb Z^+$ is a zero-measure subset of $\mathbb Z_d$.  

But the most surprising result is due to Conway, who proved that an slightly variant of the $qn+1$ problem is undecidable \cite{conway_unpredictable_1972}. It seems that the original $3n+1$ problem lies just at the edge of the decidability/undecidability threshold, and this explains partially its hardness. Since in this paper we claim that the Collatz conjecture is true, its decidability automatically follows. However, in section \S \ref{sec:decidability}, we discuss the possible undecidability of the more general $qn+1$ algorithm.

In what follows we present the formal statement of the Collatz general problem and we introduce some basic concepts and notation used throughout this paper. After that, we show the most important results of this work and we sketch their proofs.       
 
%% ------------------------------------------------------------------------------------
\subsection{Statement of the Collatz general problem}\label{subsec:statement}
%% -------------------------------------------------------------------------------------    

%% Definition of the Collatz general problem
\begin{definition}\label{def:general_problem}
(The collatz general problem $qn+1$, with $q$ any odd number). 
It consists of an algorithm that asks for a given positive integer number $n$ whether is odd or even. Then applies the following generalized Collatz-type function \cite{matthews_generalized_2010}
\begin{align}\label{eq:T_function}
	T_q(n) = \left\{ \begin{tabular}{ll} 
		$\cfrac{n}{2}$ &   if $n$ even\\[5mm]
		$\cfrac{qn+1}{2}$ & if $n$ odd 
	\end{tabular} \right.
\end{align}
and keeps iterating (indefinitely) on the result. Thus, we obtain an infinite sequence of numbers, defined as the \emph{sequence} (or orbit) of the $qn+1$ algorithm with \emph{seed} (initial condition) $n_0$, 
		 \[ SQN_{q}(n_0) = \{n_0, n_1, \ldots, n_j, \ldots \} = \{T_q^j(n_0)\}_{j \in \mathbb N}, \]
		where $n_j = T_q^j(n_0) = \overbrace{T_q \circ T_q \cdots \circ T_q}^{j}(n_0)$, $n_j \in \mathbb Z^{+}$ for all $j \in \mathbb N$, and $T_q^0 \equiv \text{Id}$ (by usual convention). The first $k \in \mathbb Z^+$ members of the sequence constitute the \emph{truncated sequence of length $k$}, and is denoted by
			\[ SQN^k_{q}(n_0) = \{n_0, n_1, \ldots, n_{k-1}\} = \{ T_q^j(n_0) \}_0^{k-1}.\]
\end{definition}

The Collatz function $3n+1$ is easily seen as the particular case $q=3$ in that generalized Collatz-type function. Regarding notation, $\mathbb Z^{+}$ is the set of positive integers, $\mathbb N = \{0,1,2,\ldots\}$ is the set of natural numbers, $\circ$ accounts for the composition of functions, capital letters are used for functions and sequences, calligraphic letters are used for sets, bold face is used for vectors, and $\#\mathcal{A}$ means the cardinality of the set $\mathcal A$. Other particular notation used along the text will be conveniently addressed when necessary. Thus, for instance,
\begin{align} 
	&SQN_3(2)=\{2,1,2,1\ldots \} \quad &SQN_3(7)=\{7,11,17,26,13 \ldots \} \nonumber \\ 
  &SQN^5_5(7)=\{7,18,9,23,58\} \quad &SQN^6_7(6)=\{6,3,11,39,137,480\} \nonumber
\end{align}

If $q=1$, the problem is trivial. There exists a unique fixed point $n=1$, and all the sequences are convergent to it. Consequently, only odd values $q\geq 3$ will be considered hereafter.

%% --------------------------------------------------------------------------------------------------------
\subsection{The conjugacy map and the modified Collatz general function}\label{subsec:conjugacy}
%% --------------------------------------------------------------------------------------------------------

We follow a numerical theoretic perspective to study the $qn+1$ problem, although some key concepts of dynamical systems theory are used as well. One of such concepts, the conjugacy of two dynamical systems, is a cornerstone in this work.

%% The conjugacy map
\begin{lemma}\label{lemma:conjugacy}
	Given an odd number $q\geq 3$, define the conjugacy map, 
	\begin{align}\label{eq:X_trans}
		X_q: \mathbb Z^{+} \longrightarrow &\mathbb Z_{cq}=\left\{ x\in \mathbb Z^{+} \; :\;\; x \equiv 1 \mod 2(q-1)  \right\} \nonumber \\
				n \longrightarrow &x=X_q(n)=2(q-1)n+1 
	\end{align}
	and the modified Collatz general function
	\begin{equation}\label{eq:modified_Collatz_function}
		F_q(x)=\cfrac{q^{\alpha_q(x)}}{2} (x+1),  
	\end{equation}
	where $\alpha_{q}(x)$ is a parity-type function (odd/even) defined as	
	\begin{equation}
		 \alpha_{q}(x) =
		\left\{ \begin{tabular}{ll} 
			$1$ &  if\; $x-1\equiv 2(q-1)$ \quad $\mod 4(q-1)$ \\
			$0$ &  if\; $x-1\equiv 0$ \hspace{14mm} $\mod 4(q-1)$
		\end{tabular}. \right.
	\end{equation}		
Then, for all odd numbers $q\geq 3$, we have $T_q \equiv X_q^{-1} \circ F_q \circ X_q$, where $X_q^{-1}$ is the inverse map of $X_q$. Consequently, the dynamical properties of $T_q(n)$ are the same as $F_q(x)$, and the sequence $S_q(x_0)=\{x_0, x_1, \ldots, x_j, \ldots\} = \{F_q^j(x_0)\}_{j \in \mathbb N}$ is equivalent to $SQN_q(n_0)$.
\end{lemma}

\noindent Thus, for instance, the counterparts of the sequences in the previous example are
\begin{align*} 
	&S_3(X_3(2))=\{9,5,9,5\ldots \} &S_3(X_3(7))=\{29,45,69,105,53 \ldots \} \\ 
  &S^5_5(X_5(7))=\{57,145,73,185,465\} &S^6_7(X_7(6))=\{73,37,133,469,1645,5761\} 
\end{align*} 

\begin{proof}
The function $T_q(n)$ could be expressed in a more compact form as
\[ T_q(n)=\frac{n}{2}+\left((-1)^{n+1} + 1 \right)\frac{((q-1)n+1)}{4}=\frac{q^{\alpha_N(n)}\left((q-1)n+1\right)-1}{2(q-1)},
\]
where the $\alpha_N$-parity function is defined as 
	\[
		\alpha_N(n) =
		\left\{ \begin{tabular}{ll} 
			$1$ & if $n\equiv 1$ \quad $\mod 2$ \\
			$0$ & if $n\equiv 0$ \quad $\mod 2$
		\end{tabular}. \right.
	\]

\noindent Taking into account the (invertible) affine map \ref{eq:X_trans}, we have that
\begin{align*}
		X_q(T_q(n)) &= 2(q-1)T_q(n)+1 = \cfrac{q^{\alpha_N(n)}}{2}(X_q(n) +1) \\ 
				&= \cfrac{q^{\alpha_q(X_q(n))}}{2}(X_q(n) +1) = F_q(X_q(n)) 
\end{align*}
\end{proof}

Thus, henceforth, we focus our study on the modified Collatz general function $F_q(x)$ which produces truncated sequences $S_q^k(x_0)=\{F_q^j(x_0)\}_0^{k-1}$ of arbitrarily large $k\in \mathbb Z^+$ lengths for every seed $x_0 = X_q(n_0)$.   

%% ------------------------------------------------------------------------
\subsection{The parity sequence}\label{subsec:parity}
%% ------------------------------------------------------------------------

Another key concept is the \emph{parity sequence} $A(q;x_0)$, which is a binary string (sequence of ones and zeros) that assembles all the evaluations of the parity function $\alpha_q(x)$ along the successive iterates of $F_q(x_0)$. As in the case of the sequences $S^k_q(x_0)$, the first $k$ members of that infinitely long parity sequence $A(q;x_0)$ define the \emph{truncated parity sequence} of length $k$ or \emph{parity vector} of length $k$, $A^k(q;x_0) \equiv \mathbf A_k = (\alpha^0, \ldots, \alpha^{k-1})$, $\alpha^j \equiv \alpha_q(x_j)$, where the dependence on $q$ and on $x_0$ is made implicit, when possible, to ease the notation. The following definitions will be also helpful.

\begin{definition}\label{def:parity}
		~				
	\begin{enumerate}[i)]		
		\item \emph{Partial parity of the sequence} $S_q^k(x_0)$: \\ Given $i,j \in \mathbb N$, $0\leq i \leq j < k$, $|\mathbf A_k|_{i}^{j} = \sum_{s=i}^{j} \alpha_q(x_s)$.
		\item \emph{Total parity} of the sequence $S_q^k(x_0)$: \\ $P_k = |\mathbf A_k|_{0}^{k-1} = \sum_{j=0}^{k-1} \alpha_q(x_j), \quad 0\leq P_k \leq k$.
		\item \emph{Parity coefficient (or ones-ratio)} for the sequence $S_q^k(x_0)$: \\ $\mu_k = \frac{P_k}{k}= \frac{1}{k} \sum_{j=0}^{k-1} \alpha_q(x_j)$, $0\leq \mu_k \leq 1.$
	\end{enumerate}
\end{definition}

\noindent The following lemma extends previous results about the parity sequence \cite{terras_stopping_1976,everett_iteration_1977,lagarias_3x_1985,matthews_generalized_2010}.

%% Parity sequence properties
\begin{lemma}\label{lemma:parity_sequence}
	Given an odd number $q\geq 3$, two integer numbers $x_0, y_0 \in \mathbb Z_{cq}$, and any $k \in \mathbb Z^+$, the following conditions hold:
	\begin{enumerate}[i)] 
		\item $x_0$ and $y_0$ have the same parity vector of length $k$ if and only if $x_0 \equiv y_0 \mod (q-1)2^{k+1}$.		
		\item $A(q;x_0)=A(q;y_0)$ if and only if $x_0 = y_0$.
		\item There exists a bijection between the set of seeds $\mathcal X_k =\{X_q(n_0) : \; n_0=1,2,3,\ldots, 2^k\}$ and the set of parity vectors of length $k$, $\mathcal A_k = \{0, 1\}^k= \underbrace{\{0, 1\} \times \cdots \times \{0, 1\}}_{k}$. 			
	\end{enumerate}	
\end{lemma}   

This result shows that the dynamics of $F_q(x)$ is encoded through the parity sequence. The relevance of the parity sequence comes, then, from its uniqueness for each seed $x_0$, which has important implications for periodicity conditions, and for its utility to 'show' the asymptotic behavior of the sequence. This will be addressed in detail in sections \S \ref{sec:periodic} and \S \ref{sec:asymptotic}, respectively. 

\begin{proof}
	We begin with claim $(i)$. For each $k \in \mathbb Z^{+}$ and any odd constant $C$, let be $y_0 = x_0 + C(q-1)2^{k+1}$. Then, 
	\[ \frac{y_0 - 1}{4(q-1)} = \frac{x_0 - 1}{4(q-1)} + C 2^{k-1}. \]
	Thus $\alpha_q(y_0) = \alpha_q(x_0) = \alpha^0$. If $k > 1$, let assume $\alpha_q(y_i) = \alpha_q(x_i) = \alpha^i$ for all $0\leq i<k-1$ and
	\[ y_i = x_i + C q^{|\mathbf A_k|_{0}^{i-1}} (q-1) 2^{k+1-i}. \]
	Then,
	\begin{align}
	 y_{i+1} &= \frac{q^{\alpha^i}}{2} (x_i + 1) + \frac{q^{\alpha^i}}{2} \left( C q^{|\mathbf A_k|_{0}^{i-1}} (q-1) 2^{k+1-i} \right) \nonumber \\
			&= x_{i+1} + C q^{|\mathbf A_k|_{0}^{i}} (q-1) 2^{k-i}, \nonumber 
	\end{align}
	and,
	\[ \frac{y_{i+1} - 1}{4(q-1)} = \frac{x_{i+1} - 1}{4(q-1)} + C q^{|\mathbf A_k|_{0}^{i}} 2^{k-(i+2)}. \]
	Thus, $\alpha_q(y_{i+1}) = \alpha_q(x_{i+1}) = \alpha^{i+1}$, and therefore, $A^k(q;y_0) = A^k(q;x_0) = \mathbf A_k = (\alpha^0, \ldots, \alpha^{k-1})$.
	
	Conversely, let assume $A^k(q;y_0) = A^k(q;x_0) = \mathbf A_k = (\alpha^0, \ldots, \alpha^{k-1})$. Using proposition \ref{prop:formula_xk}, we have for all $k \in \mathbb Z^{+}$	
	\[ x_k = \frac{q^{P_k}}{2^k}x_0 + \frac{1}{2^k}\sum_{j=0}^{k-1} 2^j q^{|\mathbf A_k|_{j}^{k-1}}, \quad \quad y_k = \frac{q^{P_k}}{2^k}y_0 + \frac{1}{2^k}\sum_{j=0}^{k-1} 2^j q^{|\mathbf A_k|_{j}^{k-1}}. \]
	Subtracting the former from the latter we obtain
	\[ y_k - x_k = \frac{q^{P_k}}{2^k} (y_0 - x_0). \]
	Hence $y_0 \equiv x_0 \mod 2^k$. Since $x_0, y_0 \in \mathbb Z_{cq}$, $y_0$ must be equivalent to $x_0$ modulo $(q-1)2^{k+1}$.
	
	Now, given that the parity sequence is infinitely long (no stopping criterion), claim $(ii)$ follows directly from 
		\[ y_k = x_k + C 2 (q-1) q^{P_k}, \quad k \in \mathbb Z^{+},\]
 \noindent where the truncated sequences $S^k_q(x_0) = \{ x_j \}_0^{k-1}$ and $S_q^k(y_0) = \{ y_j \}_0^{k-1}$ have the same parity vector $\mathbf A_k=(\alpha^0, \ldots, \alpha^{k-1})$. Then, the next parity evaluation comes from
 \[ \frac{y_k-1}{4(q-1)} = \frac{x_k-1}{4(q-1)} + \frac{1}{2} C q^{P_k}, \]
and $\alpha_q(y_k) \neq \alpha_q(x_k)$ unless $C=0$, that is, $y_0 = x_0$.

 Finally, for claim $(iii)$, let $n_0$ be a natural number in the interval $[1, 2^k]$. The seed $X_q(n_0)=x_0$ leads to the parity vector $A^k(q;x_0) \in \mathcal A_k$. Now, assume that there exists a natural number $m_0 \neq n_0$ in the interval $[1, 2^k]$ such that the seed $X_q(m_0)=y_0$ has the same parity vector of $x_0$, $A^k(q;y_0)=A^k(q;x_0)$. Applying claim $(i)$, we have that $y_0 \equiv x_0 \mod (q-1)2^{k+1}$, which leads to $|y_0 - x_0|\geq 2^{k+1}(q-1)$, where $|\cdot|$ refers to the standard absolute value. But $|X_q(m_0)-X_q(n_0)|<2(q-1)(2^k-1)$, and we get a contradiction. Therefore, every seed in $\mathcal X_k$ has a different parity vector of length $k$. Since the cardinality of $\mathcal A_k$ is precisely $2^k$, from a counting argument there is a one-to-one correspondence between the sets $\mathcal X_k$ and $\mathcal A_k$.  
\end{proof}

\begin{remark}\label{remark:parity_sequence}
	If there exist $x, y \in \mathbb Z_{cq}$ that belong to the same periodic sequence, their parity sequences are identical but shifted an amount $\ell \in \mathbb Z^+$ such that $F_q^{\ell}(x)=y$.  
\end{remark}

%% ------------------------------------------------------------------------------------
\subsection{Closed formulas for the forward iterates}\label{subsec:formulas}
%% ------------------------------------------------------------------------------------

%% Formulas
\begin{proposition}\label{prop:formula_xk}
	Given an odd number $q \geq 3$ and any $k \in \mathbb Z^+$, the $k-th$ forward iterate of the sequence $S_q(x_0)$ has two equivalent expressions,
	\begin{flalign}\label{eq:xk_1}
			x_k &= F_q^k(x_0)=\frac{q^{P_k}}{2^k}x_0 + \frac{1}{2^k}\sum_{j=0}^{k-1} 2^j q^{|\mathbf A_k(x_0)|_{j}^{k-1}} 
	\end{flalign}
	and
	\begin{equation}\label{eq:xk_2}
			x_k = F_q^k(x_0)=x_0\frac{q^{P_k}}{2^k} \prod_{j=0}^{k-1} \left(1 + \frac{1}{x_j} \right).			
	\end{equation}
\end{proposition}

These formulas are very useful to deal with arbitrary sequences. In particular, expression \ref{eq:xk_1} is well fitted to analyze periodic sequences and, in the case of $q=3$, is equal to the previous results of B\"ohm and Sontachi \cite{bohm_existence_1978} and Crandall \cite{crandall_3x1_1978}, although in their formulations the dependence on the parity vector is not explicit. Expression \ref{eq:xk_2} is related to the previous history of the sequence, and has important implications for its asymptotic behavior. Indeed, it's easily seen that for all sequences

\begin{equation}\label{eq:lower_limit}	
	\frac{F_q^k(x_0)}{x_0} > \left(\frac{q^{\mu_k}}{2}\right)^k. 
\end{equation}

\begin{proof}
Both formulas result by induction, taking into account in equation \ref{eq:xk_2} that for all $x_j \in S_q(x_0)$, we have the recurrence relation
\[x_{j+1}=F_q(x_j)=x_j\left(1+\frac{1}{x_j}\right)\frac{q^{\alpha_q(x_j)}}{2} \]
\end{proof} 

%% ------------------------------------------------------------------------
\subsection{Main results of this work}\label{subsec:main_results}
%% ------------------------------------------------------------------------

The results are divided in two blocks: \emph{periodic sequences} and \emph{the asymptotic behavior of the sequences}. Each block comprises half the way through the demonstration of the Collatz conjecture, the main goal of this paper. We begin this exposition with some definitions and a fundamental result concerning the number of periodic sequences of $F_q(x)$. 

\begin{definition}\label{def:cycle} (Cycle)
	Periodic sequence of period $p\in \mathbb Z^{+}$,
	\[ PS^p_q(x_0) = \left\{x_0, x_1,\ldots, x_{p-1}\right\}, \text{ where } x_{j+p}=x_j \in \mathbb Z_{cq}, \;\; 0\leq j \leq p-1. \]
	The period $p$ is defined as the minimum possible and is called \emph{the prime period}.	We also say that the function $F_q(x)$ has a cycle of length $p\in \mathbb Z^{+}$ at $x_0$, and $F_q^{j+p}(x_0)=F_q^{j}(x_0)$. Without loss of generality, the seed $x_0$ of the cycle is taken as the absolute minimum of the sequence. 	
\end{definition}

\begin{remark}
The parity coefficient $\mu_p = P_p/p$ of the periodic sequence $PS^p_q(x_0)$ is an irreducible fraction, i.e. $P_p$ and $p$ are co-prime.
\end{remark}

\begin{definition}\label{def:trivial_cycle}
	Trivial cycle. It is a periodic sequence of period $p$ that starts at $x_0 = X_q(1) =2q-1$. In other words, $F^p_q(2q-1)=2q-1$  (or equivalently $T^p_q(1)=1$) for some $p \in \mathbb Z^{+}$. All the other periodic sequences are called non-trivial cycles.
\end{definition}

The following theorem answers affirmatively the conjecture that, for each $q$, there are finitely many cycles in the $qn+1$ problem \cite{matthews_generalized_2010}. 

%% Main theorem for periodic cycles
%%------------------------------------%%
\begin{theorem}\label{theorem:main_periodicity}
	For each odd number $q\geq 3$, the function $F_q(x)$ has finitely many cycles. These cycles are distributed in $q-1$ families, and the seed $x_{0}$ of each cycle is
	\begin{equation}
	  x_{0}(\lambda_h) = X_q(h + \lambda_h q), \quad h \in \{1, 2, \ldots, q-1 \}, 
  \end{equation}
	for some $\lambda_h \in \mathbb N$. Moreover, if $h$ is even, then $\lambda_h$ must be odd, and vice versa.
\end{theorem}
%%------------------------------------%%

%But also reduces dramatically the number of possibilities, bringing us closer to the full resolution in the case $q=3$, since only one more cycle could exist in addition to the trivial one.
 
\vspace{0.5cm}

\noindent \textbf{Sketched proof}. The complete proof of the theorem is in section \S \ref{subsec:number_cycles}. Here we only highlight the most important ideas. 
\begin{enumerate}
	\item For a given $q\geq 3$, and from proposition \ref{prop:formula_xk} applied to a periodic sequence $PS^p_q(x_0)$, we have a Diophantine exponential equation to solve, which we call the \emph{first periodicity condition},
	%% Periodicity condition 1
	%% ------------------------	
	\begin{equation}\label{eq:periodicity}
			\left(2^p - q^{P_p} \right) x_0 = \sum_{j=0}^{p-1} 2^j q^{|\mathbf A_p(x_0)|_{j}^{p-1}}, \quad p, P_p \in \mathbb Z^{+}, \quad p >1.
	\end{equation}
	
	The unknowns are the seed of the cycle $x_0$, and the parity vector $\mathbf A_p(x_0)$.
	\item The parity vector has the structure:\\ $\mathbf A_p(x_0)=(\alpha^0 = 1,\alpha^1,\alpha^2,\ldots,\alpha^{s-1}=1, 0 \ldots, 0)$, where $1 \leq s \leq p-1$ is taken such that $|\mathbf A_p(x_0)|_{j}^{p-1} = 0$ for all $j\geq s$. This is justified because the seed $x_0$ is taken as the absolute minimum of the cycle.
	\item The 'tail of zeros' of the parity vector imposes resolvability conditions on the Diophantine exponential equation. The seed $x_0$ must belong to a family of $q-1$ congruent classes modulo $2q(q-1)$ (it is excluded the residual class zero).
	\item Given two periodic sequences in the same congruent class, they are related by a parameter $\lambda \in \mathbb N$.
	\item In each congruent class, there exists a one-to-one correspondence between the parity coefficient $\mu$ of a periodic sequence and the parameter $\lambda$. 
	\item Finally, using lemma \ref{lemma:periodicity_limits}, the inequality 
	\[
			0 < \frac{\ln 2}{\ln q} - \mu(\lambda) < \frac{1}{A + B \lambda}
	\]
	for suitable constants $A, B \in \mathbb Z^+$, imposes an upper bound for $\lambda$.
\end{enumerate}

\begin{remark}\label{remark:q5_periodicity}
 To illustrate theorem \ref{theorem:main_periodicity}, we show the interesting case $q=5$. The function $F_5(x)$, equivalently $T_5(n)$, has three known periodic sequences, those listed in table \ref{table:cycles_5}. These cycles are the 'smallest' representatives of their congruent classes. The function $F_5(x)$ could have, at most, one more representative cycle for the congruent class 33 mod 40.

%\vspace{1cm}
 
	\begin{table}[ht!]
	\caption{Known cycles of $F_5(x)$. It could exist, at most, one more representative cycle for the congruent class 33 mod 40.}
	\label{table:cycles_5}
	\resizebox{\textwidth}{!}{%  
	\begin{tabular}{c c c c c c}
		\hline 				 
		$x_{0}$ & $p$ & $P_p$ & $s$ & $T_5(n)$ & $F_5(x)$ \\
		\hline
		 $X_5(1 + 0 \cdot 5)$ & 5 & 2 & 2 & $\{1,3,8,4,2\}$ & $\{9,25,65,33,17\}$ \\
		 $X_5(2 + 3 \cdot 5)$ & 7 & 3 & 5 & $\{17,43,108,54,27,68,34\}$ & $\{137,345,865,433,217,545,273\}$ \\
		 $X_5(3 + 2 \cdot 5)$ & 7 & 3 & 3 & $\{13,33,83,208,104,52,26\}$ & $\{105,265,665,1665,833,417,209\}$ \\
		 $X_5(4 + \lambda_4 \cdot 5)$ &  &  &  & Unknown & \\
		\hline
		\end{tabular}}				
	\end{table}
\end{remark}

Another important result concerns the values of $q$ that are \emph{Mersenne numbers}, $q=2^p-1$, primes or composite. In that case, we are able to completely solve all the cycles of $F_q(x)$, which is a fundamental step towards the resolution of the Collatz conjecture.

%% Cycles for Mersenne primes
%% -----------------------------------
\begin{theorem}\label{theorem:Mersenne_cycles}
	The function $F_q(x)$ with $q=2^p-1$ a Mersenne number, has only one cycle, the trivial one with period $p$ and parity coefficient $\mu_p = 1/p$. 
\end{theorem}

The proof is in section \S \ref{subsec:Mersenne_cycles}. The main idea is the use of a \emph{second periodicity condition} \ref{eq:periodicity_2}, which is another Diophantine exponential equation, similar to the first one, but with a different arrangement. Then,  relating any non-trivial cycle with the trivial one, we deduce that there isn't exist a valid seed for a non-trivial cycle.

\vspace{0.5cm}

Theorem \ref{theorem:Mersenne_cycles} takes advantage of the special form of the Mersenne numbers, and apparently can not be generalized further. The remaining cases where $q$ are not Mersenne numbers are, consequently, much more difficult to solve, and does not seem that a general approach can be used. This is related to the \emph{tenth Hilbert's problem}, although in our case, our Diophantine exponential equation has some structure that can be exploited via the parity vector. In this way, some general conditions on the existence of cycles are found for $q$:

\begin{itemize}
	\item (Lemma \ref{lemma:periodicity_limits}). If $F_q(x)$ has a cycle with parity coefficient $\mu_p$, then $0<2-q^{\mu_p}<1/q$.
	\item (Lemma \ref{lemma:Q_cycles}). If $F_q(x)$ has a cycle in the congruent class $h \in \{1, 2, \ldots,\\ q-1 \}$, then $q$ is a divisor of $M=h2^m-1$ for some $m\in\mathbb Z^+$ that depends on the parity vector.	
	\item (Proposition \ref{prop:non_existence_cycles}). If $q$ is a composite number, then $F_q(x)$ doesn't have any cycle in the congruent classes $h=md$, where $d$ is any divisor of $q$ different from one and $q$, and $m\in \mathbb Z^+$ is any multiplier such that $md \leq q-1$.			
\end{itemize} 

The proofs of these claims are in section \S \ref{sec:periodic}. Roughly speaking, we can say that the presence of cycles is generally rare. Lemma \ref{lemma:periodicity_limits} shows also that if $F_q(x)$ has a cycle, then $q \sim 2^{p/P_p}$, and if there are multiple cycles for the same $q$, they have approximately the same parity coefficient, being $P_{p_1}/p_1 \sim P_{p_2}/p_2$ for any pair of cycles.

\vspace{0.5cm}
     
Now, we analyze the asymptotic behavior of the sequences. The following lemma is a fundamental result that bounds the sequences between two exponential functions for some special values of $q$.

%% Exponential bounds of the sequences
\begin{lemma}\label{lemma:bounds}
For each Mersenne number $q=2^p-1$, any $x_0 \in \mathbb Z_{cq} - \{2q-1 \}$, and $k \in \mathbb Z^{+}$ sufficiently large,  
	\begin{equation}\label{eq:lemma_bounds}		
		\left(\frac{q^{\mu_k}}{2}\right)^k < \frac{F^k_q(x_0)}{x_0} \leq q^{(\mu_k-\frac{1}{p})k}.
	\end{equation}
Also, for the special case $q=5$,
	\begin{equation}
		\left(\frac{5^{\mu_k}}{2}\right)^k < \frac{F^k_5(x_0)}{x_0} \leq 5^{(\mu_k-\frac{2}{5})k}.
	\end{equation}
\end{lemma}
 
The proof is in section \S \ref{subsec:proof_lemma_bounds}. The main idea is the use of formula \ref{eq:xk_2} and a procedure to upper bound the productory. This procedure uses the lowest periodic values of the function $F_q(x)$, and 'reallocates' the factors of the productory to build and upper bound. It is worth noting that the upper bound is sharp when the sequence is periodic.
    
\vspace{0.5cm}

Lemma \ref{lemma:bounds} opens a way to attack the convergence/divergence behavior of the sequences. The most important results are distilled in the following theorems.

%% Equiparity theorem
%% --------------------------------------------
\begin{theorem}\label{theorem:equiparity} (Equiparity theorem). For $q=3$, any sequence $S_3(x_0) = \{F^j_3(x_0)\}_{j \in \mathbb N}$, with $x_0 \in \mathbb Z_{c3}$, has an asymptotic parity coefficient  	
		\[ \mu_{\infty} = \lim_{k \rightarrow \infty} \frac{1}{k} |\mathbf A_k|_0^{k-1} = \frac{1}{2}.\]	
\end{theorem}

%% Divergence theorem
%%------------------------------------
\begin{theorem}\label{theorem:divergence}
For all odd $q\geq 5$, the set 
\[ \mathcal V_q := \{ n \in \mathbb Z^{+} \; : \; \text{the sequence } S_q(X_q(n))=\{F_q^j(X_q(n))\}_{j \in \mathbb N} \text{ is divergent }\} \] 

has natural density 

\[ D(\mathcal V_q) := \lim_{t \rightarrow \infty} \frac{1}{t} \#\left\{ n \in \mathcal V_q : n \leq t \right\} = 1. \]
\end{theorem}

Sections \S \ref{subsec:proof_equiparity_theorem} and \S \ref{subsec:proof_theorem_divergence} are devoted to proof these theorems. Both proofs use the same probabilistic model, but it is worth noting that theorem \ref{theorem:equiparity} is not a kind of a probabilistic result but a strong result for all sequences.

\vspace{0.5cm}

\textbf{Sketched proof of theorem \ref{theorem:equiparity}}. We highlight the most important steps.
\begin{enumerate}
	\item For each $x_0 \in \mathbb Z_{c3}$ the sequence $S_3(x_0)$ has an asymptotic parity coefficient $\mu_{\infty}$.  We do not know whether  $\mu_{\infty}$ it is the same or not for every seed $x_0$.
	\item Lemmas \ref{lemma:bounds} and \ref{lemma:purely_divergent} bound the values of the asymptotic parity coefficient for any sequence, $1/2 \leq \mu_{\infty} < 1$.
	\item We build a probabilistic model where the discrete random variable $M_k$ is the parity coefficient of a truncated sequence $S^k_3(x_0)$. The seeds are taken randomly from the set $\mathcal X_k = \{X_3(n_0) : n_0 = 1,2, 3, \ldots, 2^k\}$. Using lemma \ref{lemma:parity_sequence}, the parity vector for each seed has density $2^{-k}$, thus the probability mass function of $M_k$ is of a binomial type,
	\[ 
		\mathbb P(M_k=\mu)= \frac{1}{2^k} \binom{k}{k\mu}, \quad k\mu \in \mathbb N.
	\]
	\item Using step 2 and the symmetry of the probability mass function of $M_k$, we have
		\[ \mathbb P \left( \lim_{k \rightarrow \infty} \left| M_k - \frac{1}{2} \right| \geq \epsilon \right) = 0 \]
		for any $\epsilon>0$.
	\item Finally, the only sequences outside that probabilistic behavior are ultimately cycles, and using theorem \ref{theorem:Mersenne_cycles}, the only possible cycle has the parity coefficient $\mu = 1/2$.
\end{enumerate}

\noindent Now, we are in conditions to state the result that has driven this work. 

%% Collatz Conjecture
%% --------------------------------------------
\begin{corollary}\label{corollary:collatz}
	(Collatz conjecture). The function $F_3(x)$ is always convergent to the trivial cycle $\{X_3(1),X_3(2)\}=\{5,9\}$ for all $x \in \mathbb Z_{c3}$.
\end{corollary}

\begin{proof}
 It is a direct consequence of the equiparity theorem. From theorem \ref{theorem:Mersenne_cycles} we have that the only possible cycle of $F_3(x)$ is the trivial one, $\{X_3(1),X_3(2)\}=\{5,9\}$. Then, let assume that there exists some $x_0 \in \mathbb Z_{c3} - \{5, 9\}$ such that for all $k \in \mathbb Z^+$,
  \[ x_0 < F^k_3(x_0) = x_0 \left(\frac{3^{\mu_k}}{2} \right)^k \prod_{j=0}^{k-1} \left(1 + \frac{1}{x_j} \right) , \]
where we have used equation \ref{eq:xk_2}. Since $x_j = F_3^j(x_0)\geq x_0 > 9$ for all $j < k$, then, using theorem \ref{theorem:equiparity}, we have

\[ x_0 < F_3^k(x_0) < x_0 \left(\frac{3^{\mu_k}}{2} \right)^k \left(1 + \frac{1}{9} \right)^k = x_0 \left(\frac{3^{\mu_k} 5}{9} \right)^k < x_0 \quad \text{as } k \longrightarrow \infty, \]

and we get a contradiction. Therefore, for all $x_0 \in \mathbb Z_{c3} - \{5, 9\}$, $F_3^k(x_0) < x_0$ for some finite $k \in \mathbb Z^+$, which is called the \emph{stopping time} \cite{lagarias_ultimate_2011}. Taking $y_0 = F_3^k(x_0)$, then there exists another finite stopping time $\ell \in \mathbb Z^+$ such that $F_3^{\ell}(y_0) < y_0$. Thus, we can keep iterating since there exists some $m \in \mathbb Z^+$ such that $F_3^{m}(x_0) \in \{ 5, 9\}$. 	
\end{proof}

Next sections explain in more detail the results showed so far. In section \S \ref{sec:periodic} we deal with periodic sequences. In section \S \ref{sec:asymptotic} we analyze the asymptotic behavior of the sequences. Finally, in the last section, we discus the decidability of the algorithm $qn+1$. 

%-----------------------------------------------------------
%-----------------------------------------------------------
\section{Periodic sequences}\label{sec:periodic}
%----------------------------------------------------------
%-----------------------------------------------------------
The first obvious result about the periodic sequences is that the discrete dynamical system $qn+1$, $q\geq 3$, has no fixed point, since the equation $F_q(x)=x$ has no integer solutions in $\mathbb Z_{cq}$. Therefore we must deal with periodic sequences (cycles) of length greater than one. 

There are many previous results in the original $3n+1$ problem about the necessary conditions to have cycles. In contrast, the results are much less for the general case $qn+1$\cite{simons_non-existence_2007}. Steiner showed in \cite{steiner_theorem_1977} that an special cycle called \emph{non trivial circuit} is impossible for the $3n+1$ problem, and extended this result to $q=7$ and $q=5$ (the last one having only one \emph{non trivial circuit}) \cite{steiner_qx_1981,steiner_qx_1981-1}. Also, it was shown in \cite{eliahou_3x1_1993} and subsequent works that for the $3n+1$ problem, a non-trivial cycle must fulfill minimum conditions on its length and on the proportion of odd and even numbers. The refinements of these conditions culminated in the work of Simons and de Weger \cite{simons_theoretical_2005}, which stated that a non-trivial cycle for the $3n+1$ problem must have a length grater than $8.3485\cdot 10^{15}$, with more than $5.2673\cdot 10^{15}$ odd numbers. These remarkable results came mostly from transcendence theory and the theory of continued fractions.

Despite using powerful techniques of number theory like the ones above, it was not known whether the $3n+1$ algorithm has finitely many cycles or not. It was conjectured to be the case \cite{lagarias_ultimate_2011}, also for the general case $qn+1$ \cite{matthews_generalized_2010}, and in the following section we solve that question. But another problem we must face is to determine where are these cycles, and this seems intractable in general. Thankfully, $q$ values that are Mersenne numbers can be completely analyzed, showing a strong connection between them and the (non) existence of cycles.       

%--------------------------------------------------------------------------------------
\subsection{The number of cycles of $F_q(x)$}\label{subsec:number_cycles}
%--------------------------------------------------------------------------------------

Theorem \ref{theorem:main_periodicity} is the main result concerning periodic sequences for the general $qn+1$ problem. It states that, in general, $F_q(x)$ has finitely many cycles for each odd number $q\geq 3$. These cycles are grouped in $q-1$ congruent classes modulo $q$ in $\mathbb Z_{cq}$, and the seeds $x_0$ of the cycles in each class are in arithmetic progression. But first, we present an important lemma needed to prove the theorem.   

%% Periodicity limits
%% ---------------------------------------------------------
\begin{lemma}\label{lemma:periodicity_limits}
	The parity coefficient of a periodic sequence $PS^p_q(x_0)=\left\{x_j\right\}_0^{p-1}$ is
	\begin{equation}\label{eq:parity_coefficient_log}
		\mu_p = \frac{1}{p} \sum_{j=0}^{p-1} \log_q \left(\frac{2x_j}{x_j + 1} \right),
	\end{equation}
	and fulfills the following bounds, 
	\begin{equation}\label{eq:parity_coefficient_limits} 
	\frac{\ln ((2q-1)/q)}{\ln q} \leq \log_q \left( \frac{2x_m}{x_m + 1}\right) < \mu_p < \log_q \left( \frac{2x_M}{x_M + 1}\right) < \log_q 2 = \frac{\ln 2}{\ln q},
	\end{equation}
	where $x_m = \inf \{PS^p_q(x_0)\}$ and $x_M = \sup \{PS^p_q(x_0)\}$. Therefore, if $F_q(x)$ has a cycle with parity coefficient $\mu_p$, then 
	\begin{equation}\label{eq:cycle_q_limits}
		0<2-q^{\mu_p}<1/q.
	\end{equation}
\end{lemma}

\begin{proof}
	From equation \ref{eq:xk_2} with $x_p = x_0$, and taking logarithms, we have
	\[ P_p \ln q = \sum_{j=0}^{p-1} \ln \left(\frac{2 x_j}{x_j + 1} \right), \]
	\[ \mu_p = \frac{P_p}{p} = \frac{1}{p \ln q} \sum_{j=0}^{p-1} \ln \left(\frac{2 x_j}{x_j + 1} \right).\]
	Since
	
	\[ \frac{2 x_m}{x_m + 1} \leq \frac{2 x_j}{x_j + 1} \leq \frac{2 x_M}{x_M + 1} < 2 \quad \forall j, \quad 0\leq j \leq p-1,  \]
	
	and $x_m \geq 2q-1$, we get the bounds of equation \ref{eq:parity_coefficient_limits}. Finally, a different rearrangement of the bounds leads to
	
	\[ \ln \left(1 - \frac{1}{2q} \right) < \ln \left( \frac{q^{\mu_p}}{2} \right) < 0, \]
	
	and equation \ref{eq:cycle_q_limits} follows from taking the exponential in the above equation.
\end{proof}

%%---------------------------------------------------------------
%% Proof of main periodicity theorem
%%---------------------------------------------------------------
\paragraph{\textbf{Proof of Theorem \ref{theorem:main_periodicity}.}}
	Let assume that we have a periodic sequence of period $1<p \in \mathbb Z^+$ for some seed $x_0 = X_q(n)= 2(q-1)n + 1$, $n \in \mathbb Z^{+}$. Using equation \ref{eq:xk_1}, with $k=p$ and $x_p=x_0$, the parity vector of the sequence must fulfill the Diophantine exponential equation \ref{eq:periodicity}. We call this equation the \emph{first periodicity condition}. 
	
	The parity vector of that periodic sequence starts at the infimum of the sequence, $x_0$, thus $\alpha_q(x_0)=1$. After that, the different values $x_j \in PS_q^p(x_0)$, $1\leq j <p-1$, can increase or decrease multiple times, but the last term $x_{p-1}$ always decrease, since $x_p$ is the absolute minimum by periodicity. Therefore, in the parity vector, there exists an index $s \in \mathbb Z^+$, $1 \leq s \leq p-1$, such that $\alpha_q(x_s)=\alpha_q(x_{s+1})=\cdots=\alpha_q(x_{p-1})=0$. Rearranging terms in equation \ref{eq:periodicity}, we have,
	\[ 2^{p+1}n - 2^p + \sum_{j=s}^{p-1} 2^j = q 2^{p+1} n - 2(q-1)q^{P_p} n - 2q^{P_p} - \sum_{j=1}^{s-1} 2^j q^{|\mathbf{A}_p(x_0)|_{j}^{p-1}}, \]
	where $1 \leq s \leq p-1$ is taken such that $|\mathbf{A}_p(x_0)|_{j}^{p-1} = 0$ for all $j\geq s$. The right hand side of the above equation is proportional to $2q$, thus the left hand side should also be proportional to $2q$, and the cycle must verify
	\[ 2^{p+1} n - 2^p + \sum_{j=s}^{p-1} 2^j = 2 q m, \quad m \in \mathbb N, \]
	which leads to the equation
	\begin{equation}\label{eq:cycle}
		2^p n - q m = 2^{s-1}.	
	\end{equation}
	Its general solution is an affine map,
	\begin{equation}
	\begin{aligned}
		n &= n_0 + q \lambda \nonumber \\
		m &= m_0 + 2^p \lambda \nonumber
	\end{aligned} \;\;, \quad \text{with} \;\; 2^p n_0 - q m_0 = 2^{s-1} \quad \text{ the fixed solution,}
	\end{equation}
	and $\lambda \in \mathbb Z$ a free parameter. Since $x_0$ is the absolute minimum of the sequence, $n$ should also be a minimum, and $\lambda$ must be positive or equal zero, i.e., $\lambda \in \mathbb N$. Moreover, $n_0$ must be a residual class modulo $q$, but also $n_0 \geq 1$, thus $n_0 \equiv h \mod q$ and $h \in \{1, 2, \ldots, q-1 \}$. This can be summarized defining $n$ in each residual class as
		\[ n_{h} = h + \lambda_h q, \quad h = 1, 2, \ldots, q-1\]
for some $\lambda_h \in \mathbb N$. Thus, we have that all the possible seeds (minimums) for a given cycle in a congruent class $h$ are
	 \[ x_{0}(\lambda_h) = X_q(n_h) = X_q(h) + 2 q (q-1) \lambda_h, \quad \lambda_h \in \mathbb N. \]
	
	In addition, since every $x_{0}(\lambda)$ must be a minimum of the sequence, e.i. $\alpha_N(n_{h}) = 1$, $\lambda_h$ must be odd if $h$ is even, and Vice versa.
		
	Now, for some $h^* \in \{1, 2, \ldots, q-1 \}$, let $PS_q^{\ell} = \{X_0, X_1, \ldots, X_{\ell-1}\}$ be a periodic sequence of length $\ell$, with $X_0=X_q(h^* + \gamma q)$ its seed, $\gamma \in \mathbb N$ a fixed constant, and $\mathbf A_{\ell}(X_0) \equiv \mathbf A_{\ell} =(\alpha^0, \alpha^1, \ldots, \alpha^{\ell-1})$ its parity vector. Consider a new periodic sequence of length $r$, $\hat{PS}_q^r (\lambda) = \{x_0(\lambda), x_1(\lambda), \ldots, x_{r-1} (\lambda) \}$, with $\lambda = \lambda_{h^*}$ a free parameter, and its parity vector $\mathbf A_r(x_0(\lambda)) \equiv \hat{\mathbf A}_r =(\hat{\alpha}^0, \hat{\alpha}^1, \ldots, \hat{\alpha}^{r-1})$. These two sequences satisfy $\alpha^0 = \hat{\alpha}^0 = 1$ and $\alpha^{\ell-1} = \hat{\alpha}^{r-1} = 0$. In general $r\neq \ell$, but we can extend both sequences to have the same length $p = LCM(\ell,r)$ yet conserving their periodicity. Therefore, to simplify the analysis, we restrict ourselves to periodic sequences of length $p$, that is $\ell = r = p$. Using equation \ref{eq:periodicity} for the periodic sequence $PS_q^{p}(X_0)$, we have
\[
	\left(2^p - q^{|\mathbf A_p|_{0}^{p-1}} \right) X_0 = \sum_{j=0}^{p-1} 2^j q^{|\mathbf A_p|_{j}^{p-1}}, \quad 1< p \in \mathbb Z^{+}.
\]
Analogously, for the periodic sequence $\hat{PS}_q^{p}(x_0(\lambda))$, we have 
\[
	\left(2^p - q^{|\hat{\mathbf A}_p|_{0}^{p-1}} \right) x_0(\lambda) = \sum_{j=0}^{p-1} 2^j q^{|\hat{\mathbf A}_p|_{j}^{p-1}}, \quad 1 < p \in \mathbb Z^{+}.
\]

If both sequences have the same parity coefficient, $\hat{\mu}_p = \mu_p$, then they must fulfill the relation
\[ 
		x_0(\lambda) \sum_{j=0}^{p-1} 2^j q^{|\mathbf A_p|_{j}^{p-1}} = X_0 \sum_{j=0}^{p-1} 2^j q^{|\hat{\mathbf A}_p|_{j}^{p-1}},
\]	
which can be stated as
\[
		\sum_{j=0}^{p-1} 2^j \phi_j = 0, \quad \quad \phi_j = x_0(\lambda) q^{|\mathbf A_p|_{j}^{p-1}} - X_0 q^{|\hat{\mathbf A}_p|_{j}^{p-1}}, \quad j=0, \ldots, p-1. 
\]

This is a kind of a weighted average of the sequence $\{\phi_j\}_0^{p-1}$. The last terms of the sum have the largest contribution on the average when $p$ is sufficiently large. But, we can ensure this, adding an arbitrarily number of periods to the sequence (the length of the resulting sequence will be a multiple of $p$). This suggests that for some $m \in \mathbb Z^+$, $m \leq p-1$, $\phi_j = 0$ for all $j \geq m$. In particular, $\phi_{p-1}=0$, which means that $x_0(\lambda) = X_0$, and both sequences $PS_q^{p}(X_0)$ and $\hat{PS}_q^{p}(x_0(\lambda))$ must be identical. 

Therefore, there is a one-to-one relation between each admissible $\lambda$ of a periodic sequence $\hat{PS}_q^{p}(x_0(\lambda))$ and its parity coefficient $\hat{\mu}_p$ ($\mu_p$ corresponds to the case $\lambda = \gamma$). 	

This means that for a given residual class $h^*$, there exists an arithmetic function $\mu = \mu (\lambda)$ that relates uniquely the parity coefficient with the parameter $\lambda$. Then, applying lemma \ref{lemma:periodicity_limits} to the sequence $\hat{PS}_q^{p}(x_0(\lambda))$ and the inequality $\ln (1 + x_j^{-1}) < x_j^{-1}$ for all $j=0, \ldots, p-1$, we can bound that arithmetic function $\mu (\lambda)$ by
\[
		0 < \frac{\ln 2}{\ln q} - \mu(\lambda) < \frac{1}{p \ln q} \sum_{j=0}^{p-1} \frac{1}{x_j(\lambda)} < \frac{1}{\ln q} \frac{1}{(X_q(h^*) + 2q(q-1) \lambda)}.
\] 
	
	Thus, there must exist some finite $\lambda^*\in \mathbb N$ such that all the possible cycles $\hat{PS}_q^{p}(x_0(\lambda))$ in the congruent class $h^*$ have a $\lambda < \lambda^*$, otherwise we get a contradiction. This finally proves that there are finitely many admissible parity coefficients for each congruent class $h \in \{1, 2, \ldots, q-1 \}$.	 	
	\qed
	
\begin{remark}
	In principle, the above arguments will also work extending $\mathbb Z_{cq}$ to the negatives, which corresponds to the dynamical system $qn-1$ in the positives via the conjugation \ref{lemma:conjugacy}.
\end{remark}	

%-----------------------------------------------------------------------------------------------------
\subsection{The cycles of $F_q(x)$ for $q$ a Mersenne number}\label{subsec:Mersenne_cycles}
%-----------------------------------------------------------------------------------------------------
We begin with a lemma that states a second periodicity condition for a sequence. After-that, we introduce a necessary condition to have a cycle. Finally, we use both results to prove theorem \ref{theorem:Mersenne_cycles}. 

%% Periodicity condition 2
%% ---------------------------------------------------------
\begin{lemma}\label{lemma:second_periodicity}
(Second periodicity condition). The function $F_q(x)$ has a periodic sequence $PS^p_q(X_q(n_0))$ with parity vector $\mathbf A_p(X_q(n_0)) \equiv \mathbf A_p$, $n_0 \in \mathbb Z^{+}$, if and only if
	\begin{equation}\label{eq:periodicity_2}
	\left(2^p - q^{P_p}\right)n_0 = \sum_{j=0}^{P_p-1} 2^{g(j)}q^j, \quad p, P_p \in \mathbb Z^{+}, \quad p >1,
\end{equation}
where the arithmetic function 
\begin{align*}
		g_{\mathbf A_p} \equiv g: \mathbb N \cap [0,P_p-1] &\longrightarrow \mathbb N \cap [0,p-2] \\
		j &\longrightarrow g(j)
\end{align*} 
searches the position of each number one in the parity vector as follows:
	\[ |\mathbf{A}_p|_{g(j)}^{p-1} = j+1 \quad \text{ and } \quad |\mathbf{A}_p|_{g(j)+1}^{p-1} = j. \]
\end{lemma}

\begin{proof}
From the first periodicity condition, equation \ref{eq:periodicity}, with $x_0 = X_q(n_0) = 2(q-1)n_0 +1$, and rearranging terms, we have

\[ 2(2^{p} - q^{P_p}) n_0 = \frac{(q^{P_p}-1)}{(q-1)} + \sum_{j=0}^{s-1} 2^j \frac{(q^{|\mathbf{A}_p|_j^{p-1}}-1)}{(q-1)}, \] 

where $s \in \mathbb Z^+$ is taken such that $|\mathbf{A}_p|_{j}^{p-1} = 0$ for all $j\geq s$, $1 \leq s \leq p-1$. After some algebra, the above equation results 

\[ (2^p - q^{P_p}) n_0 = \sum_{k=0}^{P_p - 1} q^k + \sum_{j=1}^{s-1} 2^{j-1} \sum_{k=0}^{|\mathbf{A}_p|_j^{p-1}-1} q^k. \]

	It is easy to verify that $g(0)=s-1$. The function $g(n)$ is monotonically decreasing. Its lowest value is $g(P_p-1)=0$, and its highest one is precisely $g(0)=s-1$. Then,
	
\begin{align*} 
	\sum_{j=1}^{g(0)} 2^{j-1} \sum_{k=0}^{|\mathbf{A}_p|_j^{p-1}-1} q^k &= \sum_{j=1}^{g(P_p-2)} 2^{j-1} \sum_{k=0}^{P_p-2} q ^k + \sum_{j=g(P_p-2)+1}^{g(P_p-3)} 2^{j-1} \sum_{k=0}^{P_p-3} q^k + \\
	&~ \quad \quad \cdots + \sum_{j=g(1)+1}^{g(0)} 2^{j-1} \\
	&= \sum_{k=0}^{P_p - 2} q^k \sum_{j=1}^{g(k)} 2^{j-1} = \sum_{k=0}^{P_p - 2} q^k(2^{g(k)}-1). \nonumber
\end{align*}

\noindent Finally,

\[ (2^p - q^{P_p})n_0 = \sum_{k=0}^{P_p - 1} q^k + \sum_{k=0}^{P_p - 2} q^k(2^{g(k)}-1) = \sum_{k=0}^{P_p-1} 2^{g(k)}q^k. \]

\end{proof}

%% Necessary condition for cycles
%% -----------------------------------
\begin{lemma}\label{lemma:Q_cycles}
	If $F_q(x)$ has a cycle of period $p$ for the seed 
	\[x_{0} = X_q(h + \lambda_h q) = 2(q-1)(h + \lambda_h q) + 1,\quad h=1,\ldots,q-1 ,\]
then $q$ must be a divisor of $M=h2^k-1$ for some positive integer $k=p-g(0)$, where the function $g$ is defined in lemma \ref{lemma:second_periodicity}. Furthermore, if $h=1$ (trivial cycle) or $h$ is a power of 2, then $q$ must be a Mersenne number or a factor of a Mersenne number $M_m=2^m-1$, $1 < m \in \mathbb Z^{+}$.
\end{lemma}

%% --------------------------------------

\begin{proof}
Let assume that $F_q(x)$ has a cycle of period $p$, with parity vector 
\begin{align*}
	\mathbf A_p=(1,\alpha^1, \alpha^2,\ldots,\alpha^{s-2},\alpha^{s-1}=1, \alpha^s = 0, \ldots, \alpha^{p-1}=0),& \\
	\quad 1 \leq s \leq p-1, \quad 1 \leq P_p \leq s,& 
\end{align*} 	
where, as before, $s \in \mathbb Z^+$ is taken such that $|\mathbf A_p|_{j}^{p-1} = 0$ for all $j\geq s$. Then, from equation \ref{eq:periodicity_2}, we have
\[
	2^{g(0)}(h2^{p-g(0)}-1) = h q^{P_p} - \lambda_h q (2^p-q^{P_p}) + \sum_{j=1}^{P_p-1} 2^{g(j)} q^{j},
\]
where $k=p-g(0)=p-s+1$ is a positive integer. Since the right hand side of the above equation is proportional to $q$, $q$ must be a divisor of $M=h2^k-1$. Finally, if $h=2^a$ for some $a \in \mathbb N$, then $M=2^{k+a}-1$, which is a Mersenne number for the exponent $m=k+a$.
\end{proof}

%% Descend procedure for Mersenne numbers q
%% -------------------------------------------
\begin{lemma}\label{lemma:Mersenne_residual_class}
	Let $q$ be a Mersenne number, $q= 2^{m} -1$, $1 < m \in \mathbb Z^+$. If $q$ divides $k2^c-1$ for some $c \in \mathbb Z$ and some $k \in \mathbb Q$ such that $k2^c \geq 1$, then $k=2^d$ for some $d \in \mathbb Z$, $d<m$.		
\end{lemma}

\begin{proof}
	Let $c = am+b$, for some constants $a,b \in \mathbb Z$, $0\leq b < m$. Since $q$ divides $k2^c-1$, there must exist some $\sigma_0 \in \mathbb N$, such that
	\begin{equation}\label{eq:Mersenne_class}
		\sigma_0 (2^m-1) = k 2^{b}2^{am}-1.
	\end{equation}
We define $K_0 = k 2^{b}2^{am}$. If $K_0 < 2^m$, $\sigma_0 = 0$ and $k=2^{-c}$. If $K_0 \geq 2^m$, then we define $K_1 = k 2^{b}2^{(a-1)m}\geq 1$, and equation \ref{eq:Mersenne_class} results
\[ 
	\sigma_0 (2^m-1) = K_1(2^m-1) + K_1-1.
\]
Therefore, for a valid solution, $2^m-1$ also divides $K_1-1$, and there must exist some $\sigma_1 \in \mathbb N$ such that
\[
		\sigma_1 (2^m-1) = K_1-1.
\]
Since $K_1< K_0$, this defines a descend procedure that can be iterated until $K_{\ell} = k 2^{b}2^{(a-\ell)m} < 2^m$ for some step $\ell \in \mathbb Z^+$, and equation \ref{eq:Mersenne_class} has a valid solution if and only if
\[
	\sigma_{\ell} (2^m-1) = K_{\ell}-1.
\]
In that case, the only possible solution is $\sigma_{\ell}=0$ and $K_{\ell}=1$, which means that $k=2^{\ell m - c}$.
\end{proof}

\noindent Now, we are in conditions to prove the other main result concerning periodic cycles.\\

%%----------------------------------------------------------------
%% Cycles for Mersenne numbers
%% ---------------------------------------------------------------
\paragraph{\textbf{Proof of Theorem \ref{theorem:Mersenne_cycles}.}}
Let $q$ be a Mersenne number, $q= 2^{p} -1$, prime or composite. Using lemma \ref{lemma:conjugacy}, the sequence $S_q=\{F^j_q(x_0)\}_{j \in \mathbb N}$ is equivalent to the sequence $SQN_{q}=\{T^j_q(n_0)\}_{j \in \mathbb N}$, which is better for the present analysis. Taking $n_0 = 1$ as the seed for the trivial cycle, we have
	\[ n_1=T_q(1)= \frac{q + 1}{2} = 2^{p-1}. \] 
Therefore, $T_q^p(1) = 1$, and equivalently, the function $F_q(x)$ has a trivial cycle with period $p$ and parity coefficient $\mu_p = 1/p$. 

Conversely, let assume that $F_q(x)$, with $q = 2^{k} - 1$ a Mersenne number, has a cycle for some $n_0 > 1$ and some parity vector $\mathbf A_r$ of period $1< r \in \mathbb Z^{+}$, which comes as a solution of equation \ref{eq:periodicity_2}. Let extend that cycle $k$ times, hence the periodic sequence has now a period $p=k r$ and a total parity $P_p = k P_r$, that is,
\[
	\left(2^{kr} - q^{k P_r} \right) n_0 = \sum_{j=0}^{k P_r-1} 2^{g(j)} q^j. 
\]
Analogously, we extend the trivial cycle $r$ times, verifying the equation
\[
	2^{kr} - q^{r} = \sum_{j=0}^{r-1} 2^{(r-j-1)k} q^j. 
\]  
Combining both equations, we have that
\[
2^{g(0)}\left(n_0 2^{(r-1)k-g(0)} - 1 \right) = \left(q^{k P_r} - q^r \right) n_0 + \sum_{j=1}^{k P_r - 1} 2^{g(j)} q^j - n_0 \sum_{j=1}^{r-1} 2^{(r-j-1)k} q^j.
\]  
Since the RHS of the above equation is proportional to $q$, $q$ divides the term $n_0 2^{(r-1)k-g(0)} - 1$. Finally, applying lemma \ref{lemma:Mersenne_residual_class}, $n_0 = 2^d$ for some $d \in \mathbb Z$. But, in order to be $n_0$ a valid seed for a non-trivial cycle, $n_0$ must be odd and greater than one, leading to the final contradiction.\qed

%----------------------------------------------------------------------------------------------------
\subsection{General conditions for the existence of cycles}\label{subsec:conditions_cycles}
%----------------------------------------------------------------------------------------------------
As mentioned before, the search of the cycles of the function $F_q(x)$ is, in general, out of reach for the techniques used in this work. However, in addition to lemma \ref{lemma:Q_cycles}, we highlight below some general conditions that cycles must fulfill in order to exist. For completion, we show in table \ref{table:cycles} the known non-trivial cycles (up to our knowledge) of the function $F_q(x)$ for different values of $q$ \cite{steiner_qx_1981}. 
\begin{table}[htbp]
	\caption{Known non-trivial cycles of $F_q(x)$ for different values of $q$.}
	\label{table:cycles}	
		\begin{tabular}{c c c c c}
		\hline\noalign{\smallskip} 				 
		 $q$ & $p$ & $P_p$ & $n_{h}$ & $x_{0}$ \\
		\noalign{\smallskip}\hline\noalign{\smallskip} 
		 $5$ & 7 & 3 & $2+3\cdot 5 =17$ & 137 \\
		 $5$ & 7 & 3 & $3+2\cdot 5 =13$	& 105 \\
		$181$ & 15 & 2 & $27+0\cdot 181 =27$ & 9721\\
		$181$ & 15 & 2 & $35+0\cdot 181 =35$ & 12601\\
		\noalign{\smallskip}\hline
		\end{tabular}			
	\end{table}

% Non-existence of cycles in congruent classes
%----------------------------------------------
\begin{proposition}\label{prop:non_existence_cycles}
If $q$ is a composite number, then $F_q(x)$ doesn't have any cycle in the congruent classes $h=md$, where $d$ is any divisor of $q$ different from one and $q$, and $m\in \mathbb Z^+$ is any multiplier such that $md \leq q-1$. 
\end{proposition}

\begin{proof}
All the possible cycles are solutions of the \emph{second periodicity condition}, equation \ref{eq:periodicity_2}. As $q$ is a composite number, we take $d$ as a non-trivial divisor of $q$, that is, $d \neq 1, q$. If there exists some $m \in \mathbb Z^+$ such that $md\leq q-1$, then we have

\[ \left(2^p - q^{P_p}\right)d(m + \lambda_h (q/d)) = 2^{g(0)} + \sum_{j=1}^{P_p-1} 2^{g(j)}q^j, \quad p, P_p >1, \]
  
where, from theorem \ref{theorem:main_periodicity}, the seed must be $n_h = h + \lambda_h q$. Since $q$ is odd, $d$ never divides $2^{g(0)}$, and the above equation doesn't have any valid solution. Therefore, there isn't exist any valid seed $n_{h}$ for a cycle in the congruent class $h=md$.
\end{proof}

%% Cycles with parity one
%%-------------------------------------------------------
\begin{proposition}\label{prop:cycle_parity_one}
	The function $F_q(x)$ has a cycle with total parity $P_p=1$ if and only if $q$ is a Mersenne number.
\end{proposition}

\begin{proof}
From the \emph{second periodicity condition}, equation \ref{eq:periodicity_2}, with $P_p = 1$, we have

\[ (2^p - q)n_0 = 1. \]

The unique possible solution is $n_0 = 1$ and $q= 2^p - 1$, a Mersenne number.
\end{proof}

%% Trivial cycles
%%-------------------------------------------------------
\begin{proposition}\label{prop:trivial_cycle}
	The function $F_q(x)$ has a trivial cycle with total parity $P_p=2$ if and only if $q=5$.
\end{proposition}

\begin{proof}
We restrict ourselves to the search of trivial cycles, which are solutions of equation \ref{eq:periodicity_2} with $n_0=1$, that is 
\begin{equation}\label{eq:periodicity_trivial}
	2^p = q^{P_p} + \sum_{j=0}^{P_p-1} 2^{g(j)}q^j, \quad p, P_p \in \mathbb Z^{+}, \quad p >1,
\end{equation}

If $P_p=2$, equation \ref{eq:periodicity_trivial} reads $2^p = q^2 + q + 2^{g(0)}$, which means that there exists some odd $\sigma_0 \in \mathbb N$ such that

\begin{equation*}
	\sigma_0 = \frac{2^{p-g(0)}-1}{q} = \frac{q+1}{2^{g(0)}} = \frac{2^{p-g(0)} + q}{2^{g(0)} + q}, 
\end{equation*}
or
\begin{equation}\label{eq:sigma_0}
	\sigma_0 \left(2^{g(0)} \sigma_0 - 1 \right) = 2^{p-g(0)} - 1. 
\end{equation}

Obviously, for a valid solution, $g(0)\geq 1$, $p-2g(0)\geq 0$, and $2^{g(0)}$ divides $\sigma_0 - 1$. Hence, the change of variable $\sigma_0 = 1 + 2^{g(0)}\sigma_1$ for some $\sigma_1 \in \mathbb N$ leads to

\[ \sigma_1(2^{2g(0)}\sigma_1 + 2^{g(0)+1}-1) = 2^{p-2g(0)}-1. \]

If $\sigma_1 = 0$, then $p-2g(0) = 0$ and $q=2^{g(0)}-1$. But, in that case, $q$ is a Mersenne number with a trivial cycle of period $p = g(0)$ (theorem \ref{theorem:Mersenne_cycles}), a contradiction. Therefore, $p-4g(0)\geq 0$ and $2^{g(0)+1}$ divides $\sigma_1-1$. Thus, in the above equation, the change of variable $\sigma_1 = 1 + 2^{g(0)+1}\sigma_2$ for some $\sigma_2 \in \mathbb N$ results 

\[ \sigma_2 \left(2^{3g(0)+1}\sigma_2 + 2^{2g(0)+1} + 2^{g(0)+1} - 1\right) = 2^{p-3g(0)-1} - 2^{g(0)-1} - 1. \]

If $\sigma_2 = 0$, then $g(0)=1$ ($s=2)$, $p=5$, and $q=5$, which leads to the trivial cycle

\[ \{X_5(1)=9, F_5(9)=25, F_5(25)=65, F_5(65)=33, F_5(33)=17\}. \]

If $\sigma_2$ is a positive odd number, then $2^{g(0)-1}$ divides $\sigma_2 - 1$, and again, the change of variable $\sigma_2 = 1 + 2^{g(0)-1}\sigma_3$ for some $\sigma_3 \in \mathbb N$, results

\begin{align*} 
	\sigma_3 \left(2^{4g(0)}\sigma_3 + 2^{3g(0)+2} + 2^{2g(0)+1} + 2^{g(0)+1} - 1\right) = \\
 2^{p-4g(0)} - 2^{2g(0)+2} - 2^{g(0)+2} - 2^2 - 1. 
\end{align*}

Therefore, we repeat this argument as many times as necessary, where $0 \leq \sigma_{j+1} < \sigma_j < \sigma_0 $ for all $j \in \mathbb Z^+$, until $\sigma_n = 0$ for some $n \in \mathbb Z^+$. Then, the above equation results,
\[
	 1 + \sum_{j=1}^{m-1} 2^{(a_j g(0) + b_j)} = 2^{p-(a_m g(0) + b_m)},
\]
for some $m \in \mathbb Z^+$, $n \leq m$, and for some sequences $\{a_j\}_1^{m}$ and $\{b_j\}_1^{m}$ of integer numbers such that $\forall j$, $p > a_{j+1} g(0) + b_{j+1} > a_j g(0) + b_j \geq 0$.

Finally, from a parity (odd/even) argument, we have that
\[
	\sum_{j=1}^{m-1} 2^{(a_j g(0) + b_j)} = 1,
\]
thus $m=2$, $a_1= b_1=0$, and $p = 1 + a_2 g(0) + b_2$. Since $n \leq 2$, $\sigma_1 = 0$ or $\sigma_2 = 0$, which are the cases already analyzed. Therefore, the unique possible solution of equation \ref{eq:periodicity_trivial} for $P_p = 2$ is $q=5$, $p=5$, and $g(0)=1$.
\end{proof}
%----------------------------------------------------------------------- 

\begin{remark}\label{remark:trivial_cycles}
The general case $P_p\geq 3$ is much more difficult to analyze and seems intractable. My guess is that equation \ref{eq:periodicity_trivial} with $P_p\geq 3$ has no valid solutions for $q$, hence there are no more trivial cycles (see conjecture \ref{conjecture_trivial_cycles} in the next section).
\end{remark}

%-----------------------------------------------------------------------------------------------
\subsection{Conjectures about cycles}\label{subsec:conjectures_cycles}
%-----------------------------------------------------------------------------------------------
Along this work arose questions that remain open. This section groups these questions in the form of conjectures, with the hope that further development could solve them. Before going to their description, we present a useful definition.  

\begin{definition}\label{def:counting_function}(Counting function)
The function $\pi(F_q)$ counts the number of cycles of the function $F_q(x)$ for all $x \in \mathbb Z_{cq}$.
\end{definition}

\noindent From theorem \ref{theorem:main_periodicity} we have that, for each odd number $q\geq 3$, $\pi(F_q)$ is finite and bounded by a function that depends on $q$. Bellow, we list three main conjectures about the counting function and the trivial cycles. 

\begin{conjecture}
	There are infinitely many values of $q$ such that $\pi(F_q)=0$. 
\end{conjecture}

\begin{conjecture}
	There exists some $q*$ such that $\frac{\pi(F_q)}{q}\ll 1$ for all $q > q*$.  
\end{conjecture}

\begin{conjecture}\label{conjecture_trivial_cycles}
	The function $F_q(x)$ has a trivial cycle only in the cases $q=5$ and $q=2^p-1$, $p\geq 2$.
\end{conjecture}

%-------------------------------------------------------------------------------------
%-------------------------------------------------------------------------------------
\section{Asymptotic behavior of the sequences}\label{sec:asymptotic}
%-------------------------------------------------------------------------------------
%-------------------------------------------------------------------------------------

Equation \ref{eq:xk_2} is the starting point in the study of the asymptotic behavior of the sequences of the function $F_q(x)$. From that equation, lemma \ref{lemma:bounds} gives explicit bounds to the highest and lowest values of the sequences, which are used to proof the main results of this section, theorems \ref{theorem:equiparity} and \ref{theorem:divergence}. But first, we begin with a basic result related to the divergence of the sequences.

%% Divergent sequences

\begin{definition}\label{def:purely_divergent_seq}
	(Purely divergent sequence) For a given $q \geq 3$, the sequence $S_q(x)=\{F_q^j(x)\}_{j \in \mathbb N}$ is said to be purely divergent if $\alpha_q(F_q^j(x))=1$ for all $j \in \mathbb N$. In other words, the related parity sequence is an infinite sequence of one's
	\[ A(q;x) = \{1,1,\ldots, 1, \ldots\}. \]
\end{definition}

\begin{lemma}\label{lemma:purely_divergent}
	There isn't exist a purely divergent sequence in the algorithm $qn+1$.
\end{lemma}

\begin{proof}
	Given an odd number $q\geq 3$, let $x\in \mathbb Z_{cq}$ be the seed of a purely divergent sequence $S_q(x)=\{x, F_q(x), F^2_q(x),\ldots \}$ with the parity sequence $A(q;x)=\{1,1,\ldots \}$.	Let $y=F_q(x)$ be the next iterate of $x$. Its parity sequence $A(q;y)=\{1,1,\ldots\}$ is exactly the same as $A(q;x)$, so applying lemma \ref{lemma:parity_sequence}, we have that $x = y = q(x+1)/2$, which is impossible.
\end{proof}

In what follows, we show the proofs of lemma \ref{lemma:bounds} and theorems \ref{theorem:equiparity} and \ref{theorem:divergence}. They are essential to justify the convergence of all sequences for $q=3$ (Collatz conjecture), and the divergence of almost all sequences for $q\geq 5$ (Crandall conjecture).

%% Proof lemma Bounds for a sequence  
%% -------------------------------------------------------------------------------------------
\subsection{Proof of lemma \ref{lemma:bounds}}\label{subsec:proof_lemma_bounds}
%% -------------------------------------------------------------------------------------------

Using equation \ref{eq:xk_2} for the sequence $S^k_q(x_0)=\{x_j\}_0^{k-1}$, $x_0 \in \mathbb Z_{cq}-\{2q-1\}$, we have the lower bound
	\[ x_0 \left(\frac{q^{\mu_k}}{2}\right)^k < F^k_q(x_0) = x_0 \left(\frac{q^{\mu_k}}{2}\right)^k \prod_{j=0}^{k-1} \left(1 + \frac{1}{x_j} \right). \]
	
 Our purpose is to find the upper bound of the above productory, which we call \emph{the harmonic productory} of the sequence. From theorem \ref{theorem:Mersenne_cycles} we know that $F_q(x)$, $q=2^p-1$, has only one periodic sequence, the trivial cycle 

\[ \{X_q(1),X_q(2^{p-1}),X_q(2^{p-2}),\ldots,X_q(2)\}. \] 

In the case $q=3$ ($p=2$), it is clear that this sequence is the lowest possible one, and the harmonic productory is bounded by 
\begin{align}
	\prod_{j=0}^{k-1} \left(1 + \frac{1}{x_j} \right) &\leq \left(1 + \frac{1}{X_3(1)} \right)^{k/2} \left(1 + \frac{1}{X_3(2)} \right)^{k/2} = \frac{2^k}{3^{k/2}}. \nonumber 
\end{align}
Now, let us study the general case $q=2^p-1$ with $p \geq 3$. For convenience, let $k$ be a positive integer number such that $k > 2^{p-1} - p$, and let the set $\mathcal C_p$ be all the integer numbers between 1 and $2^{p-1}$ that are not powers of 2. Then,

\[ \mathcal C_p = \bigcup_{h=1}^{p-2} \mathcal C_h, \quad \text{ with } \mathcal C_h = \left\{ n \in \mathbb N : \; 2^h+1 \leq n \leq 2^{h+1} - 1 \right\}, \] 
and we define $m_h = \#\mathcal C_h=2^h -1$, and $m_p = \#\mathcal C_p= \sum_{h=1}^{p-2} m_h$. Hence, the harmonic productory for all $X_q(n_j)$ with $n_j \in C_h$, is bounded by

\begin{align}
	\prod_{j=1}^{m_h} \left(1 + \frac{1}{X_q(n_j)} \right) &< \left(1 + \frac{1}{X_q(2^h)} \right)^{2^h-1} . \nonumber 
\end{align}

Let the sequence $S_q^k (x_0)$ be a non-periodic sequence that takes, among others, all the values $X_q(n)$, $n \in \mathcal C_p$, only once. Reordering the sequence in ascending order, $\hat{S}_q^k (x_{(0)}) = \{x_{(j)}\}_0^{k-1}$, with $x_{(j)} < x_{(j+1)}$, it results  
\begin{align}
	\prod_{j=0}^{k-1} \left(1 + \frac{1}{x_{(j)}} \right) &< \prod_{j=m_p}^{k-1} \left(1 + \frac{1}{x_{(j)}} \right) \prod_{h=1}^{p-2} \left(1 + \frac{1}{X_q(2^h)} \right)^{2^h-1} \nonumber \\
	 &<  \left(1 + \frac{1}{X_q(2^{p-1})} \right)^{k-m_p} \prod_{h=1}^{p-2} \left(1 + \frac{1}{X_q(2^h)} \right)^{2^h-1}. \nonumber
\end{align}

Since for all $h \in \{1, \ldots, p-2\}$, $X_q(2^h)^{-1} < X_q(2^\ell)^{-1}$, $0 \leq \ell<h$, we can safely redistribute as convenience the factors of the productory from those of higher powers to the lower ones, yet conserving the inequality's sign. Therefore, using the values of the trivial cycle as a reference, the harmonic productory is bounded by
\begin{align}
	\prod_{j=0}^{k-1} \left(1 + \frac{1}{x_{j}} \right) &\leq \left(1 + \frac{1}{X_q(1)} \right)^{k/p} \cdots \left(1 + \frac{1}{X_q(2^{p-1})} \right)^{k/p}, \nonumber 
\end{align}
for any sequence $S_q^k(x_0) = \{x_j\}_0^{k-1}$, $x_0 \in \mathbb Z_{cq}-\{2q-1\}$. The equality holds only when the sequence belongs to the trivial cycle (but starting at $x_0 \neq 2q-1$). 

Now, for a Mersenne number $q=2^p-1$, there is the recurrence relation $X_q(2^j)+1 = 2 X_q(2^{j-1})$, $j=1,\ldots,p-1$. Therefore,
\begin{align}
	\prod_{j=0}^{k-1} \left(1 + \frac{1}{x_j} \right) &\leq \left(2^{p-1} \frac{(X_q(1) + 1)}{X_q(2^{p-1})} \right)^{k/p} = \left(2^{p-1} \frac{2q}{q^2} \right)^{k/p} = \frac{2^k}{q^{k/p}}. \nonumber 
\end{align}

\noindent Finally, the function $F_q(x)$ is bounded by
\[ \frac{F^k_q(x_0)}{x_0} \leq \left(\frac{q^{\mu_k}}{2}\right)^k \frac{2^k}{q^{k/p}} = q^{(\mu_k-\frac{1}{p})k}. \]

\noindent By the other hand, the function $F_5(x)$ has the trivial cycle 
	\[ \{X_5(1)=9,X_5(3)=25,X_5(8)=65,X_5(4)=33,X_5(2)=17 \}, \] 

thus proceeding as before, with $k>5$ and the set $\mathcal C = \{5,6,7 \}$, where $X_5(n)^{-1} < X_5(4)^{-1}$ for all $n \in \mathcal C$, the harmonic productory reads
\begin{equation}
	1 < \prod_{j=0}^{k-1} \left(1 + \frac{1}{x_j} \right) \leq \left(2^{5} \frac{5 \cdot 13 \cdot 33 \cdot 17 \cdot 9}{9 \cdot 25 \cdot 65 \cdot 33 \cdot 17} \right)^{k/5} = \left( \frac{2^5}{5^2} \right)^{k/5} = \frac{2^k}{5^{2k/5}}, \nonumber 
\end{equation}
and
\[ \frac{F^k_5(x_0)}{x_0} \leq \left(\frac{5^{\mu_k}}{2}\right)^k \frac{2^k}{5^{2k/5}} = 5^{(\mu_k-\frac{2}{5})k}.\qed \]

%% Proof of Equiparity theorem
%% --------------------------------------------------------------------------------------------------------------------------
\subsection{Proof of Theorem \ref{theorem:equiparity} (equiparity theorem)}\label{subsec:proof_equiparity_theorem}
%% --------------------------------------------------------------------------------------------------------------------------
	~\newline
		Let the sequence $S_3(x_0) = \{F^j_3(x_0)\}_{j \in \mathbb N}$ be any arbitrary sequence with parity sequence $A(3;x_0)= \{\alpha^j\}_{j\in \mathbb N}$, and the related sequence of parity coefficients $R(x_0) = \{\mu_{\ell}\}_{\ell \in \mathbb Z^+}$, where $\mu_{\ell} = (1 / \ell) \sum_{j=0}^{\ell-1} \alpha^j$ is the cumulative average of $A(3;x_0)$. The first thing we have to proof is the convergence of the sequence $R(x_0)$ for any seed $x_0 \in \mathbb Z_{c3}$, that is, the existence of the asymptotic parity coefficient $\mu_{\infty}$ (depending on the seed, in principle). We do a proof by contradiction.
	
	Let assume that the sequence $R(x_0)$ has no limit in $\mathbb R$ for some seed $x_0 \in \mathbb Z_{c3}$. This means that the sequence $S_3(x_0)$ never visits a cycle and must be ultimately divergent. The first conclusion, that never visits a cycle, follows from the contrary assumption, that is, the sequence effectively enters into a cycle and keeps looping indefinitely, thus the parity coefficient of the sequence approaches to the parity coefficient of the cycle, and the sequence $R(x_0)$ has a limit. The second conclusion, that the sequence must be divergent, follows from the fact that if there exists some finite constant $C_0$ such that all the values of $S_3(x_0)$ are upper-bounded by $C_0$, the function $F_3(x_0)$ is indefinitely exhausting values in the interval $[1,C_0] \cap \mathbb Z_{c3}$ until it repeats someone, thus defining a cycle. And we can repeat the argument for subsequent greater values of constants $C_n$. Therefore, we must analyze the possible existence of a divergent sequence $S_3(x_0)$ whose related sequence of parity coefficients $R(x_0)$ has no limit.
		
	Since $F_3^j(x_0)/x_0 \approx (3^{\mu_j}/2)^j$ (equation \ref{eq:xk_2}) for sufficiently large values of $j$ and $x_0 \in \mathbb Z_{c3}$, there must exist some $k \in \mathbb Z^+$ such that $\mu_j > \mu_k > \ln 2/\ln 3$ and $F_3^j(x_0) > F_3^k(x_0)$ for all $j>k$. Then, we make a renormalization such that $y_0 = F_3^k(x_0)$, and define the sequence $S_3(y_0) = \{F^{n}_3(y_0)\}_{n \in \mathbb N}$, which is a sub-sequence of the original $S_3(x_0)$. The sequence $S_3(y_0)$ has the related sequence of parity coefficients $R(y_0) = \{\hat{\mu}_{h} \}_{h \in \mathbb Z^+}$, which is not a sub-sequence of the original $R(x_0)$. From the simple relation
	
	\[ \frac{P_h}{h} > \frac{P_h + P_k}{h+k} > \frac{P_k}{k}, \quad j = h+ k, \]
	
we have that $\hat{\mu}_h > \mu_{h+k}$ for all $h \in \mathbb Z^+$. Since the sequence $S_3(x_0)$ is divergent, we can repeat the same argument as many times as necessary, obtaining each time a sequence of greater parity coefficients. After an infinite number of steps, we get a sequence of parity coefficients sufficiently close to 1, their highest possible value. But in that case, the sequence of parity coefficients has limit 1, which is a contradiction and an impossibility due to lemma \ref{lemma:purely_divergent}. Therefore, the sequence $R(x_0)$ must have a limit. 

Although $R(x_0)$ is a convergent sequence with limit the asymptotic parity coefficient $\mu_{\infty}$, the sequence $S_3(x_0)$ could be divergent (it couldn't reach a non-trivial cycle due to theorem \ref{theorem:Mersenne_cycles}). Note also that from lemmas \ref{lemma:purely_divergent} and \ref{lemma:bounds}, $\mu_{\infty}$ must be bounded within $[1/2, 1)$. The proof is straightforward. Assume that $\mu_{\infty} < 1/2$ (thus $x_0 \neq 5$). Taking the limit in equation \ref{eq:lemma_bounds}, we have
	\[  \lim_{k \rightarrow \infty} \frac{F_3^k(x_0)}{x_0} \leq \lim_{k \rightarrow \infty} 3^{-(1/2-\mu_{\infty})k} = 0, \] 
	
	and we get a contradiction. 
		
	Now, let us compute $\mu_{\infty}$. The question here is whether $\mu_{\infty}$ has the same value for every sequence. Since the truncated sequence $S^k_3(x_0)$ is arbitrarily long, and apparently doesn't exhibit any structure (pseudo-randomness), it is reasonable to use a probabilistic model. Rather than using the individual evaluations of the parity function as random events, it is better to work with the whole parity vectors, since we don't need to deal with the independence of the successive events. For a given $k \in \mathbb Z^{+}$ sufficiently large, let $M_k$ be a discrete random variable that means the parity coefficient of the vector $\mathbf A_k$ corresponding to the sequence $S_3^k(x_0)$, 
		
		\begin{align} 
			M_k: \; \mathcal A_k \longrightarrow &\mathcal W_k = \left\{0, \frac{1}{k}, \frac{2}{k}, \ldots, \frac{k-1}{k}, 1 \right\}  \nonumber \\
			\mathbf A_k(x_0) \longrightarrow &M_k(\mathbf A_k(x_0)) = \frac{1}{k} |\mathbf A_k|_{0}^{k-1} \nonumber		
		\end{align}
		
		where $\mathcal A_k$ is the set of all parity vectors of length $k$, and the seed $x_0$ is any (random) value in $\mathcal X_k = \{X_3(n_0) : n_0 = 1,2, 3, \ldots, 2^k\}$. As lemma \ref{lemma:parity_sequence} shows, all the parity vectors of $\mathcal A_k$ are represented with the same density $2^{-k}$. Furthermore, there are exactly $\binom{k}{k\mu}$ parity vectors with total parity $P_k = k \mu$, $\mu \in \mathcal W_k$. Therefore, the probability mass function of $M_k$ is	 
	
	\begin{equation}
			\mathbb P(M_k=\mu)= \frac{1}{2^k} \left( 
			\begin{array}{c}
				k \\
				k\mu
			\end{array} \right) = \frac{1}{2^k} \frac{k!}{(k\mu)! (k-k\mu)!}, \quad \quad k\mu \in \mathbb N, \quad \mu \in \mathcal W_k,
	\end{equation}
	
	\noindent and its moment generating function,
	
		\[ \Psi_k(t) = \frac{1}{2^k} (1 + e^{t/k})^k, \quad t \in \mathbb R,\]
		
\noindent with the expected value $\mathbb E(M_k)= \Psi'_k(0) = 1/2$ (first derivative of $\Psi_k(t)$ at zero) and the standard deviation $\sigma_k = \sqrt{\Psi''_k(0)-\mathbb E(M_k)^2} = 1/(2\sqrt{k})$. Given that the probability mass function of $M_k$ is symmetric about $1/2$ and $\lim_{k \rightarrow \infty} M_{k}$ must be bounded within $[1/2, 1)$, for any $\epsilon >0$ we have

 \[ \mathbb P \left( \lim_{k \rightarrow \infty} \left| M_k - \frac{1}{2} \right| \geq \epsilon \right) = 2 \mathbb P\left( \lim_{k \rightarrow \infty} M_k - \frac{1}{2} \leq - \epsilon \right) = 0,  \]

where $|\cdot|$ refers to the standard absolute value. Therefore, $M_k$ converges almost surely to its expected value $1/2$. 

If there exists some exceptional sequence that has an asymptotic parity coefficient different from $1/2$, it must break systematically this probabilistic behavior \emph{ad infinitum}, meaning that there is an underlying structure. This happens when the sequence ultimately enters into a cycle with parity coefficient $\mu_{\infty}$. But, by theorem \ref{theorem:Mersenne_cycles}, the only possible cycle is the trivial one, with parity coefficient precisely $1/2$. This finally proves the theorem in all cases. \qed 

\vspace{0.5cm}
It seems that it is not possible to generalize such strong result for any odd number $q$. First, there could be more unknown cycles, and second, even if they are known, such as in the case of a Mersenne number, their parity coefficient is not equal to the probabilistic one. This forces us to develop a mechanism to decide when a sequence becomes periodic and when the sequence goes on and on without an apparent structure, facing obligatorily the decidability barrier of the algorithm. However, we can conclude that for all $q\geq 5$, most of the numbers will lead to divergent sequences. It is shown in theorem \ref{theorem:divergence}, whose proof is bellow.

%%--------------------------------------------------------
%% Proof of Theorem divergence
%%-----------------------------------------------------------------------------------------------------
\subsection{Proof of Theorem \ref{theorem:divergence}}\label{subsec:proof_theorem_divergence}
%%-----------------------------------------------------------------------------------------------------
Following the same probabilistic model as in theorem \ref{theorem:equiparity}, let $M_k=|\mathbf A_k|_0^{k-1}/k$ be the random variable from the set $\mathcal A_k$ of parity vectors to the set $\mathcal W_k = \{0,1/k,2/k,\ldots, 1\} \subset \mathbb Q$. Since lemma \ref{lemma:parity_sequence} also applies for each $q \geq 5$, the probability mass function of $M_k$ is again

	\[ \mathbb P(M_k=\mu)= \frac{1}{2^k} \binom{k}{k\mu}, \quad k\mu \in \mathbb N, \quad \mu \in \mathcal W_k, \]

with expected value $\mathbb E(M_k)= 1/2$ and standard deviation $\sigma_k = 1/(2\sqrt{k})$. Using now the Chebyshev inequality, we have

	\[ \mathbb P \left( \left| M_k - \frac{1}{2} \right| \geq \eta \sigma_k  \right) \leq \frac{1}{\eta^2}, \]
	
	\noindent where $\eta > 0$ is any real number. For each $0<\epsilon<1/2$ such that $ \eta = \epsilon / \sigma_k$, we have

	\[ \mathbb P \left( \left| M_k - \frac{1}{2} \right| \geq \epsilon \right) \leq \frac{1}{4 \epsilon^2 k}. \]
	
	\noindent Using the symmetry of the probability mass function of $M_k$ and rearranging terms, we have 
	
	\[ \mathbb P \left( M_k > \frac{1}{2} - \epsilon \right) \geq 1 - \frac{1}{8 \epsilon^2 k}. \]
  
\noindent From equation \ref{eq:lower_limit}, it follows that the function $F_q(x)$ diverges when $\mu_k > \ln 2/\ln q$ for a $k$ sufficiently large. Since we are interested in computing the probability of such event, $\mathbb P( M_k > \ln2/\ln q)$, we take the value $\epsilon=1/2-\ln 2/ \ln q$ ($q \geq 5$) in the above equation, and it results 

	\[  \mathbb P \left( M_k > \frac{\ln 2}{\ln q} \right) \geq 1 - \frac{\ln^2 q}{k (2 \ln^2 q - 8 \ln 2 (\ln q - \ln 2))}. \]

  Note that $\ln 2/\ln q \notin \mathcal W_k$. 
	
	Finally, let us compute now the natural density of the set $\mathcal V_q$. We need the proportion of natural numbers in $[1,2^k]$ that lye in $\mathcal V_q$. Therefore,
	
\begin{align}
	D(\mathcal V_q) &= \lim_{t \rightarrow \infty} \frac{1}{t} \# \left\{ n \in \mathcal V_q : n \leq t \right\} = \lim_{k \rightarrow \infty} \frac{1}{2^k} \# \left\{ n \in \mathcal V_q : n \leq 2^k \right\} \nonumber \\
				&= \lim_{k \rightarrow \infty} \frac{1}{2^k} \sum_{\frac{\ln 2}{\ln q}<z \in \mathcal W_k} \binom{k}{kz} = \lim_{k \rightarrow \infty} \mathbb P \left( M_k > \frac{\ln 2}{\ln q} \right) = 1, \nonumber 
\end{align} 

where we have used from lemma \ref{lemma:parity_sequence} that for each $n \in \mathcal V_q$, there is a bijection with the related parity vectors $\mathbf A_k$ with total parity $P_k > k \ln 2/\ln q$. \qed 
	
\vspace{0.5cm}
Even with this remarkable result, it is not known any divergent sequence. This is related, again, with the decidability of an algorithm, since we don't know which number $x_0$ leads to a periodic sequence, and which number $x_0$ explodes to infinity. I discuss this in the next section.

%%---------------------------------------------------------------------------------------
%%---------------------------------------------------------------------------------------
\section{On the decidability of the Collatz general problem}\label{sec:decidability}
%%---------------------------------------------------------------------------------------
%%---------------------------------------------------------------------------------------
We begin this section with an important result due to J.H. Conway in 1972 \cite{conway_unpredictable_1972}. It states that a natural generalization of the algorithm $qn+1$ is undecidable.

\begin{theorem}\label{theorem:Conway}(Conway, 1972)
	There is no algorithm which, given an integer $n$ and a generalized Collatz function $g$ of the form
	\[ g(n) = a_i n + b_i \quad \quad if \quad n \equiv i \quad \left(\text{mod } d \right), \quad 0\leq i \leq d-1, \]
	where $2\leq d \in \mathbb N$ and the coefficients $a_i$ and $b_i$ are rational numbers such that $g(n)$ is always integral, determines whether or not there exists a positive integer $k$ such that $g^k(n)=1$.
\end{theorem}

This results warns us about the possible undecidability of the $qn+1$ problem and may be the answer to the difficulties found when looking for general patterns on the sequences. It just seems that the decidability/undecidability threshold is a barrier which prevents us to explore more deeply. 

The possible undecidability of the $qn+1$ algorithm means that the function $F_q(x)$ has computational capabilities in the sense of a Turing machine. But by the other hand, if the function $F_q(x)$ has enough amount of randomness in their successive iterates, then this would fight against an undecidability nature. This will remain an open problem, but we take part in favor of undecidability through the following conjecture.

\begin{conjecture}(Undecidability of the algorithm $qn+1$)
	There is no algorithm which, given an odd integer $q\geq 5$ and a positive integer $x \in \mathbb Z_{cq}$, determines whether or not there exists a positive integer $k$ such that $F_q^k(x)=2q-1$, where $F_q(x)$ is the modified Collatz general function defined in equation \ref{eq:modified_Collatz_function}.
\end{conjecture}

\begin{remark}\label{remark:undecidability}
This conjecture does not preclude the existence of particular cases of $q$ in which the algorithm can be decidable.  
\end{remark}

Supporting arguments to that conjecture are presented bellow. Firstly, for a given $q$, the relation $F_q^k(x)=2q-1$ defines the exponential Diophantine equation

\[ 2^k(2q-1) = q^{|\mathbf A_k|_{0}^{k-1}} x + \sum_{j=0}^{k-1} 2^j q^{|\mathbf A_k|_{j}^{k-1}}, \]

where the unknowns are $x \in \mathbb Z_{cq}$ and the parity vector $\mathbf A_k$ (with its length $k$). Thus, we must solve an exponential Diophantine equation with an arbitrary large number ($k+1$) of unknowns.

Secondly, the statistical asymptotic divergent behavior of $F_q(x)$ for $q\geq 5$ points towards an endless running of the sequence, having enough room in principle to reach any desired number. 

Finally, the iterative behavior of the function $F_q(x)$ makes plausible some kind of recurrent structure in the long run. As a final remark, the case $q=7$ seems to be a good candidate for further development in determining the decidability/undecidability of the general $qn+1$ algorithm.

% References
%--------------------------------------------
%--------------------------------------------
\bibliography{Collatz}
\bibliographystyle{plain}
%--------------------------------------------
%--------------------------------------------

\end{document}